\documentclass[11pt]{amsart}

\usepackage[left=1in, right=1in, top=1.2in, bottom=1.2in]{geometry}
\usepackage{microtype, mathtools, mathrsfs, enumerate, dsfont, esvect,centernot}
\usepackage{verbatim}
\usepackage[linktocpage]{hyperref}
\hypersetup{
    colorlinks=true,        
    linkcolor=red,          
    citecolor=blue,         
    filecolor=magenta,      
    urlcolor=cyan           
}

\usepackage{amssymb}
\usepackage{url}

\theoremstyle{plain}
	\newtheorem{theorem}{Theorem}
		\numberwithin{theorem}{section}
	
	\newtheorem{proposition}[theorem]{Proposition}

\theoremstyle{definition}
	\newtheorem{definition}[theorem]{Definition}
	\newtheorem{remark}[theorem]{Remark}
\newif\ifhascomments \hascommentstrue
\ifhascomments
  \newcommand{\dragos}[1]{{\color{red}[[\ensuremath{\bigstar\bigstar\bigstar} #1]]}}
  \newcommand{\matt}[1]{{\color{red}[[\ensuremath{\spadesuit\spadesuit\spadesuit} #1]]}}
\else
  \newcommand{\dragos}[1]{}
  \newcommand{\matt}[1]{}
\fi

\begin{document}

\title[On the importance of being primitive]{On the importance of being primitive}

\author[Jason P. Bell]{Jason P. Bell}
\address{Department of Pure Mathematics, \\University of Waterloo, \\ Waterloo, ON \\ Canada \\ N2L 3G1}
\email{jpbell@uwaterloo.ca}


\maketitle
\begin{abstract}
We give a brief survey of primitivity in ring theory and in particular look at characterizations of primitive ideals in the prime spectrum for various classes of rings.%

\end{abstract}




\section{Introduction}
Given a ring $R$, an ideal $P$ of $R$ is called \emph{left primitive} if there is a simple left $R$-module $M$ such that $P$ is the ideal of elements $r\in R$ that annihilate $M$; i.e., $$P=\{r\in M\colon rm=0~\forall m\in M\}.$$  The notion of a right primitive ideal can be defined analogously.  Bergman \cite{Ber}, while still an undergraduate, gave an example of a ring in which $(0)$ is a right primitive but not a left primitive ideal, so in general the two notions do not coincide; although, in practice, the collections of left and right primitive ideals are often the same.  For the remainder of this paper, we will speak only of primitive ideals, with the understanding that we are always working on the left, and, in any case, for most of the rings considered in this survey, a left/right symmetric characterization of primitivity is given. 

We note that if $P$ is a primitive ideal that is the annihilator of a simple left $R$-module $M$ then $M$ natural inherits a structure as a simple left $R/P$-module that is now faithful; i.e., it has zero annihilator as an $R/P$-module.  A ring for which the $(0)$ ideal is primitive is then called a primitive ring and so we see that if $P$ is primitive then $R/P$ is a primitive ring.  We note that the intersection of the primitive ideals of a ring is equal to the Jacobson radical---this is true whether one works with left or right primitive ideals \cite[Prop. 3.16]{GW}.
\subsection{Why are primitive ideals important?}
Primitive ideals have played a significant role in the development of ring theory throughout its history.  In this brief subsection, we give an overview of the reasons why primitive ideals are important.

A fundamental tool in understanding finite groups is via their representation theory.  Given a simple module $M$, by Schur's theorem $\Delta:={\rm End}_R(M)$ is a division ring and there is a map $f:R\to {\rm End}_{\Delta}(M)$.  We call such a map an \emph{irreducible representation} of $R$.  The image of $R$ is isomorphic to a subring of a ring of linear operators of a vector space over a division ring and hence one can view the image of the ring $R$ as now being a ring of linear transformations.  In general, some information is lost under this procedure and this image gives an imperfect picture of the ring $R$.  One can think of this process intuitively by imagining a sculpture and then shining a light against it and seeing the shadow that is cast. One can gain only an imperfect idea of the original form of the sculpture from its shadow, but it would not be unreasonable for one to cling to the vague hope that if one were to shine a light against the sculpture in all possible manners then one could hope to recover a good idea of its shape from the totality of the data encoded by the shadows.
  
Similarly, with irreducible representations one does not generally have all information about the ring, but one can hope that if one has a sufficiently good understanding of a large number of irreducible representations of a ring then one can answer difficult structure-theoretic questions about the ring itself.

In practice, it is very difficult to find all irreducible representations of a ring.  This can be done for group algebras of many finite groups over the complex numbers, but in general the problem is intractable.  Nevertheless, just knowing the annihilators of the simple modules will often allow one to prove non-trivial facts about a ring.  

As mentioned above, if $R$ is a primitive ring with a faithful simple $R$-module $M$ then 
by Schur's theorem $\Delta={\rm End}_R(M)$ is a division ring and the Jacobson density theorem gives that 
there is an embedding $f:R\to {\rm End}_{\Delta}(M)$ given by $f(r)(m)=r\cdot m$.  Jacobson's theorem says in fact that this embedding is dense in the following sense: if $v_1,\ldots ,v_m$ are $\Delta$-linearly independent elements of $M$ and $w_1,\ldots ,w_m\in M$ then there is some $r\in R$ such that $f(r)(v_i)=w_i$.  If one studies much of Jacobson's work on the theory of rings, one sees a unifying philosophy to his approach to many ring theoretic problems.
\vskip 2mm
{\bf Jacobson's method}
\begin{enumerate}
\item Given a problem about a ring $R$ show, if possible, that it suffices to consider $R/J(R)$, where $J(R)$ is the Jacobson radical of $R$.
\item If the first step can be completed, use the fact that $J(R)$ is an intersection of primitive ideals to obtain that $R$ is a subdirect product of primitive rings $R/P$.
\item Next, if possible, use the second step to show that one can reduce to the case when $R$ is primitive.
\item Finally, use the the embedding of a primitive ring into a ring of linear operators described above to reduce the problem to one that is essentially linear algebra.
\end{enumerate}
Although Jacobson's method is not a panacea for dealing with all problems in ring theory, it nevertheless provides an effective means of answering many types of questions that arise within the discipline.  Perhaps the best example of this method being used in practice is to obtain Jacobson's famous commutativity theorem, which says that if $R$ is a ring with the property that for every $x\in R$ there is some positive integer $n(x)>1$ such that $x^{n(x)}=x$ then $R$ is commutative.  It is not difficult to show that in such a ring the Jacobson radical must be zero and since this property is inherited by homomorphic images, it suffices to prove the primitive case.  Then one can use the Jacobson density theorem and the fact that such a ring can have no nonzero nilpotents to reduce to the case when $R$ is a division ring.  The division ring case is not immediate, but it is still not so hard to show that a division ring with this property must be a field (see Jacobson \cite{J1}).

A second reason that the primitive ideals are important can be seen via analogy with commutative algebra.  Given a commutative ring $R$, one can form an affine scheme $X={\rm Spec}(R)$, consisting of the prime ideals of $R$ endowed with the Zariski topology.  We note that if $R$ is commutative then a simple module is of the form $R/P$ with $P$ maximal and the annihilator of $R/P$ is $P$.  Thus the primitive ideals of $R$ are just the closed points of $X$.  In general primitive ideals need not be maximal, but they are always prime ideals: to see this, observe that if $P$ is a primitive ideal of $R$ and $M$ is a faithful simple left $S:=R/P$-module then if $P$ is not prime then there exist nonzero $a,b\in S$ such that $aSb=(0)$; but now since $b\neq 0$ and $M$ is faithful and simple we have $SbM=M$ and so $(0)= aSbM = aM$, which is nonzero since $a$ is nonzero, a contradiction.  Thus the primitive ideals form a distinguished subset of the prime ideals. In the commutative setting, the power of the affine scheme $X$ associated to $R$ is that one can think of $R$ as being a ring of functions on $X$.  In the noncommutative setting, this point of view is necessarily less powerful, but it is nevertheless still the case that the prime spectrum and the primitive spectrum, for many classes of rings, encode a lot of valuable information about the ring.  The focus of the remainder of this survey is to study the problem of determining which prime ideals in the prime spectrum of a ring are primitive.
\section{Important background results}
There are several necessary and sufficient conditions for prime ideals to be primitive.  We give a brief overview of these results.
\subsection{Locally closed prime ideals}
We first outline the relationship between a prime ideal being locally closed in the prime spectrum and being primitive.
\begin{definition} Let $P$ be a prime ideal of a ring $R$. We say that $P$ is \emph{locally closed} in ${\rm Spec}(R)$ if the intersection of the prime ideals that properly contain $P$ is strictly larger than $P$.  
\end{definition}
We will often just say that $P$ is locally closed, with the understanding that we are talking about being locally closed in ${\rm Spec}(R)$.

We recall that a subset $Y$ of a topological space $X$ is locally closed if $Y$ is the intersection of an open set of $X$ with a closed subset of $X$.  We observe that our definition of $P$ being locally closed coincides with the topological definition for $\{P\}$ being a closed point of the prime spectrum. 
To see this, if $I$ denotes the intersection of the prime ideals that properly contain $P$ and $I\neq P$ then $\{P\} = C_P \setminus C_I$, where given a two-sided ideal $J$ of $R$ we let $C_J$ denote the closed subset of ${\rm Spec}(R)$ consisting of the prime ideals that contain $J$.  Thus $\{P\}$ is locally closed if it is equal to the intersection of the closed set $C_P$ and the open set $C_I^{c}$.  On the other hand, if $\{P\}=C\cap U$ for some closed set $C$ and some open set $U$ then $C=C_J$ for some $J$ and $U=C_I^c$ for some ideal $I$ so $P$ is the unique prime that contains $J$ and does not contain $I$ and so we see that the intersection of the prime ideals of $R$ that properly contain $P$ is an ideal that properly contains $P$.

In the case that $R$ is noetherian, if $P$ is locally closed and one lets $I$ denote the intersection of the prime ideals that contain $P$ then $I$ is a finite intersection of prime ideals by Noether's theorem.  Hence in this setting being locally closed is equivalent to saying that the poset of prime ideals that properly contain $P$ has a finite set of minimal elements.

The connection between being locally closed and being primitive is as follows.
\begin{proposition} Let $R$ be a ring and let $P$ be a locally closed prime ideal of $R$.  Then $P$ is primitive if and only if the Jacobson radical of $R/P$ is zero.\label{lc}
\end{proposition}
\begin{proof} By replacing $R$ by $R/P$, we may assume that $P=(0)$.  Let $I\neq (0)$ denote the intersection of the nonzero prime ideals of $R$.  If $J(R)=(0)$ then since $I$ is not contained in the Jacobson radical, there exists some $x\in I$ such that $1+x$ is not a (left) unit.  Let $L=R(1+x)$.  Then $L$ is a proper left ideal of $R$ and hence by Zorn's lemma there is some maximal left ideal $J$ of $R$ that contains $L$.  Let $M=R/J$.  Then $M$ is simple and notice that if $M$ is not faithful then its annihilator must be some nonzero prime ideal $Q$ that contains $I$.  But this means $x\in Q$ and so $xM=(0)$.  But by construction $(1+x)\cdot (1+J) = 0$ since $1+x\in L\subseteq J$ and so $x$ and $1+x$ both annihilate $1+J\in M$.  But this means $1=(1+x)-x$ annihilates $1+J$ and so $1+J=0+J$, a contradiction.  Thus $M$ is faithful and so $R$ is primitive.  Conversely if the Jacobson radical of $R$ is nonzero then since the Jacobson radical is the intersection of the primitive ideals, we see that $(0)$ is not primitive and so $R$ is not primitive.
\end{proof}

There is also a general connection between a ring being primitive and the centre of the ring being small, although there are examples of primitive rings with big centres (see for example, Irving \cite{Irv01}).  A general principle, however, is that being primitive is in some sense orthogonal to being commutative. This can be seen to some degree from the fact that we showed earlier: a commutative ring is primitive if and only if it is a field.  Kaplansky \cite[Theorem 23.31]{Rowenbook} generalized this and showed that a primitive ring satisfying a polynomial ring is isomorphic to a matrix ring over a division ring with the division ring being finite-dimensional over its centre.  

\subsection{Noncommutative localization and the Nullstellensatz}
An immensely useful construction in commutative algebra is that of forming the field of fractions of an integral domain.  In general, there are many subtleties that arise when one tries to mimic this construction in the noncommutative setting, the most obvious obstruction being the fact that if one multiplies two ``left fractions,'' $s_1^{-1}r_1$ and $s_2^{-1}r_2$, together, then there is no obvious way to rewrite this product as a left fraction $s^{-1}r$.  Indeed, one cannot always do this and rings with sets of left denominators for which one can do this are said to satisfy the Ore condition.  If $R$ is a semiprime noetherian ring then a remarkable result of Goldie \cite[Theorem 6.15]{GW} shows that if we let $S$ denote the set of regular elements of $R$; that is, the elements of $R$ that are neither left nor right zero divisors of $R$, then we can localize at $S$ and form a quotient ring $S^{-1}R$ that is a semisimple Artinian ring, which we denote by $Q(R)$.  This can then be regarded as being a ``noncommutative field of fractions'' of $R$.  It is of course not a field unless $R$ is a commutative domain, but it is the next best thing: a ring that is isomorphic to a finite product of matrix rings over division rings.  In the case that $R$ is prime noetherian, the quotient ring is in fact a simple Artinian ring.

If $R$ is a prime noetherian $k$-algebra then $Q(R)$ is again a $k$-algebra and its centre, $Z(Q(R))$ is a field extension of $k$.  We call this field the \emph{extended centre} of $R$. If $R$ is a primitive noetherian ring then $Z(Q(R))$ embeds in the endomorphism ring of a faithful simple module $M$ \cite[Lemma II.7.13]{BrGo}, and so if the endomorphism ring is small then the centre of the Goldie quotient ring is necessarily small, too.  In the case that $R$ is a noetherian $k$-algebra, with $k$ a field, and $P$ is a prime ideal with the property that $Z(Q(R))$ is an algebraic extension of $k$ then we say that $P$ is \emph{rational}.

Notice that in the case that $R$ is commutative, ${\rm End}_R(M)$ is just the field $R/{\rm Ann}(M)$.  For many commutative algebras, one shows that the surjective homomorphisms onto fields are constrained via the Nullstellensatz.  More precisely, there is a strong Nullstellensatz for commutative Jacobson rings.  We recall that a ring $R$ is \emph{Jacobson} if every prime ideal is the intersection of the primitive ideals above it; equivalently, being Jacobson is the same as saying that the Jacobson radical of $R/P$ is zero for all prime ideals $P$ of $R$. Notice in light of Proposition \ref{lc}, we have that if $R$ is a Jacobson ring then a locally closed prime ideal is primitive.  For commutative rings we have the following version of the Nullstellensatz \cite[Theorem 4.19]{Eisenbud}: if $R$ is a commutative Jacobson ring and $S$ is a finitely generated $R$-algebra then $S$ is Jacobson and whenever $P$ is a maximal ideal of $S$ we have $R\cap P$ is a maximal ideal of $R$ and $S/P$ is a finite extension of $R/(P\cap R)$.  In particular, when $R$ is a field this gives the classical Nullstellensatz: that if $S$ is a finitely generated $R$-algebra and if $P$ is a maximal ideal of $S$ then $S/P$ is a finite extension of $R$.  

In analogy with the commutative setting, we say that a noncommutative noetherian $k$-algebra $R$ satisfies the Nullstellensatz if the following hold:
\begin{enumerate}
\item $R$ is a Jacobson ring;
\item if $P$ is a primitive ideal of $R$ and $M$ is a faithful simple left $R/P$-module, then ${\rm End}_{R/P}(M)$ is an algebraic $k$-algebra.
\end{enumerate}
Then it is immediate that if $R$ is a left noetherian algebra that satisfies the Nullstellensatz then we have the following 
\begin{theorem}\label{thm:null} Let $k$ be a field and let $R$ be a left noetherian $k$-algebra that satisfies the Nullstellensatz.  Then we have:
\begin{enumerate}\item a locally closed prime ideal is primitive;
\item a primitive ideal is rational.
\end{enumerate}
\end{theorem}
\begin{proof} Since $R$ is Jacobson, $J(R/P)=(0)$ and so if $P$ is locally closed, it is primitive by Proposition \ref{lc}.  If $P$ is primitive and $M$ is a faithful simple $R/P$-module then there is a $k$-algebra embedding of $Z(Q(R/P))$ into ${\rm End}_{R/P}(M)$ \cite[Lemma II.7.13]{BrGo} and so $Z(Q(R/P))$ is an algebra $k$-algebra that is an integral domain and hence it is an algebraic extension of $k$.  Thus $P$ is rational.
\end{proof}
The Nullstellensatz, as it turns out, is a rather broad phenomenon that holds in many settings.
We give an overview.
\begin{theorem} Let $k$ be a field and let $R$ be a $k$-algebra.  Then the following hold:
\begin{enumerate}
\item if there is a chain $k\subseteq R_1\subseteq \cdots \subseteq R_d=R$ of $k$-algebras such that for $i=0,1,\ldots ,d-1$, each $R_{i+1}$ is either a finitely generated left and right $R_i$-module or there is some $x\in R_{i+1}$ such that $R_{i+1}$ is generated by $R_i$ and $x$ and $R_i x + R_i = x R_i +R_i$ then $R$ satisfies the Nullstellensatz;
\item if ${\rm dim}_k(R) < |k|$ then $R$ satisfies the Nullstellensatz;
\item if $R$ is a finitely generated algebra satisfying a polynomial identity then $R$ satisfies the Nullstellensatz;
\item if $R$ is the group algebra of a polycyclic-by-finite group $G$ then $R$ satisfies the Nullstellensatz.
\end{enumerate}

\end{theorem}
\begin{proof} See Brown and Goodearl \cite[II.7]{BrGo} for (i)--(ii). 
The fact that affine PI algebras satisfy the Nullstellensatz is a result of Amitsur and Procesi (see \cite[Theorem 6.3.3]{Row2}).  Finally, Brown \cite{B82} proved the Nullstellensatz holds for the group algebra of a polycyclic-by-finite group.
\end{proof}
\subsection{Primitivity of skew polynomial rings}
\label{skew}
One of the most basic constructions of noncommutative rings comes via skew polynomial rings.  In this setting we have a base ring $R$, an automorphism $\sigma$ of $R$ and a $\sigma$-derivation $\delta$ of $R$; that is, for $a,b\in R$ we have $\delta(ab) = \sigma(a)\delta(b)+\delta(a)b$.  Then we can construct a ring $R[x;\sigma,\delta]$, called a skew polynomial ring, which, as a set, is just $R[x]$ but with multiplication extending $R$ given by
$$x\cdot r = \sigma(r) x + \delta(r).$$  In the case when $\delta=0$, it is customary to omit the $\sigma$-derivation and write $R[x;\sigma]$; and when $\sigma$ is the identity, $\delta$ is just a derivation and it is customary to write $R[x;\delta]$.  A natural question is: when is a ring of the form $R[x;\sigma,\delta]$ primitive?  For the case of a commutative noetherian base ring $R$, Goodearl and Warfield \cite{GoWa} characterized when $R[x;\delta]$ is primitive. Specializing to the case when $R$ is a finitely generated commutative domain over a field $k$ of characteristic zero, their result shows that $R[x;\delta]$ is primitive if and only if there is some maximal ideal $P$ of $R$ that does not contain a nonzero prime ideal $Q$ of $R$ that is invariant under the derivation $\delta$.  In general this is a sufficient condition to be primitive: indeed, Jordan \cite{Jor77} shows that if $R$ is a right noetherian ring with a derivation $\delta$, then $R[x;\delta]$ is primitive if there is a maximal right ideal $M$ of $R$ containing no nonzero $\delta$-invariant two-sided ideal of $R$, but in general counterexamples exist, and the Goodearl Warfield characterization is necessarily more complicated to deal with these examples.

In the case of automorphisms there are similar characterizations of primitivity.  When $R$ is a noetherian ring satisfying a polynomial identity with automorphism $\sigma$ a result of Leroy and Matzcuk \cite{LM96} characterizes when $R[x;\sigma]$ is primitive.  Given a ring $R$ with an automorphism $\sigma$, we say that $R$ is $\sigma$-\emph{special} if there is some $a\in R$ such that $r\sigma(r)\cdots \sigma^n(r)$ is nonzero for all $n\ge 1$ and such that every nonzero $\sigma$-stable ideal $I$ of $R$ contains $r\sigma(r)\cdots \sigma^n(r)$ for some $n\ge 1$, depending upon $I$.  Then Leroy and Matzcuk \cite[Theorem 3.10]{LM96} show that if $R$ is a commutative noetherian ring with automorphism $\sigma$, then $R[x;\sigma]$ is primitive if and only if $\sigma$ has infinite order and $R$ is $\sigma$-special.  Jordan \cite{Jor93} has done work on the skew Laurent case and has given an analogous characterization for commutative base rings $R$.
\section{The Dixmier-Moeglin equivalence}
One of the most appealing results in terms of characterizing primitive ideals in rings is the work of Dixmier and Moeglin \cite{Dix77, Moe80}, who showed that if $L$ is a finite-dimensional complex Lie algebra then the primitive ideals of the enveloping algebra $U(L)$ are just the prime ideals of ${\rm Spec}(U(L))$ that are locally closed in the Zariski topology. Furthermore, they proved that a prime ideal $P$ of $U(L)$ is primitive if and only if the Goldie ring of quotients of $U(L)/P$ has the property that its centre is just the base field of the complex numbers.  In general, for a field $k$ and a left noetherian $k$-algebra $R$, prime ideals $P$ for which the centre of the Goldie ring of quotients $Q(R/P)$ of $R/P$ has the property that its centre is an algebraic extension of $k$ are called \emph{rational} prime ideals.  Hence Dixmier and Moeglin's result can be regarded as saying that for primes $P$ of ${\rm Spec}(U(L))$ we have the following equivalences:
$$P ~{\bf locally~closed}\iff P~{\bf primitive}\iff P~{\bf rational}.$$
 
 In light of their work, we now say that a noetherian $k$-algebra $R$ satisfies the \emph{Dixmier-Moeglin equivalence} if we have the equivalence of the three above properties---primitivity, rationality, and local closedness---for primes in the spectrum of $R$.  The Dixmier-Moeglin equivalence is now known to be a very general phenomenon that holds for many classes of algebras beyond just enveloping algebras of finite-dimensional Lie algebras.  
 
 Some additional examples include affine PI algebras \cite[2.6]{Von1}, group algebras of nilpotent-by-finite groups \cite{Z} (see Theorem \ref{thm:nbf}), various quantum algebras \cite{GoLet} (and see \cite[II.8.5]{BrGo}), affine cocommutative Hopf algebras of finite Gelfand-Kirillov dimension in characteristic zero \cite{BL14}, Hopf Ore extensions of affine commutative Hopf algebras \cite{BSM18}, twisted homogeneous coordinate rings of surfaces \cite{BRS10}, Hopf algebras of Gelfand-Kirillov dimension two (under mild homological assumptions) \cite{GZ}, finitely generated connected graded domains of GK dimension two that are generated in degree one (cf. Artin and Stafford \cite{AS}), and even in settings where the noetherian property does not hold \cite{ABR, Lor09, Lor08} (see Section \ref{LorSec}) and generalizations involving stronger properties \cite{BLN}.

On the other hand, the equivalence is not universal and there are several finitely generated noetherian counterexamples are now known \cite{BLLM17, BCM17, Irv79, Lor}.  
 \subsection{History}
 The history of the Dixmier-Moeglin equivalence is as follows.  Dixmier \cite[Theorem C]{Dix77} proved a weaker version of what is now the Dixmier-Moeglin equivalence for enveloping algebras of complex finite-dimensional Lie algebras.  He showed in this case that for prime ideals $P$, the properties: primitivity, rationality, and the property of there being a countable set of elements $a_n\not\in P$ such that every prime ideal strictly containing $P$ must contain some $a_n$, are equivalent.  It is known that in the case of affine prime noetherian algebras over uncountable base fields, having $z\in Z(Q(R))$ that is transcendental over the base field gives rise to an uncountable set of height one prime ideals of $R$, by localizing at a suitable countable Ore set in which one obtains a ring with a non-algebraic centre and then employing a result of Jategaonkar \cite{Jat}; conversely, having an uncountable set of height one prime ideals gives that $Z(Q(R))$ is not algebraic over the base field (see \cite{Irv79a}).  This countability condition was later strengthened to the condition of $P$ being locally closed by Moeglin \cite{Moe80}.   
 
Much of the work on the Dixmier-Moeglin equivalence relates to Hopf algebras or Hopf-like algebras (e.g., bialgebras and deformations of bialgebras). Predating the Dixmier-Moeglin equivalence were results such as Kaplansky's theorem, which says that a primitive algebras satisfying a polynomial identity is isomorphic to a matrix ring over a division ring $D$ with $[D:Z(D)]<\infty$ and Zalesski\u\i's theorem \cite{Z}, which states that if $G$ is the group algebra of a finitely generated torsion-free nilpotent group then a primitive ideal $P$ of the group algebra of $G$ is maximal.  Both of these results characterize primitive ideals as being precisely the maximal ideals, and both of these algebras are known to satisfy the Dixmier-Moglin equivalence (see Theorem \ref{thm:DMEex}).  Snider \cite{Sni} showed that for a polycyclic group $G$ the group algebra of $G$ over a field that is not algebraic over a finite field has all primitive ideals maximal if and only if $G$ is nilpotent-by-finite, and Lorenz \cite{Lor} shortly thereafter gave an example of a polycyclic group that is not nilpotent-by finite whose group algebra does not satisfy the Dixmier-Moeglin equivalence.  In his example, $(0)$ is a primitive ideal of the group algebra that is not locally closed.  We now give the key result in showing that the Dixmier-Moeglin equivalence holds for affine algebras satisfying a polynomial identity as well as for certain group algebras.

\begin{theorem} Let $k$ be a field and let $A$ be a noetherian $k$-algebra and suppose that $A$ satisfies the Nullstellensatz.  If every $P$ in ${\rm Spec}(A)$ has the property that every nonzero ideal of $A/P$ intersects the centre of $A/P$ non-trivially then $A$ satisfies the Dixmier-Moeglin equivalence and all primitive ideals are maximal.\label{thm:max}
\end{theorem}
\begin{proof}
Since $A$ satisfies the Nullstellensatz, to prove the Dixmier-Moeglin equivalence it suffices to prove that rational ideals are locally closed.  If $P$ is rational then $Z(A/P)$ is a field and since every nonzero ideal intersects the centre non-trivially, we see that $A/P$ is simple.  In particular, we have rational ideals are maximal and these are necessarily locally closed.  The result follows.
\end{proof}

\begin{theorem} Let $k$ be a field and let $A$ be either a finitely generated $k$-algebra satisfying a polynomial identity or the group algebra over $k$ of a finitely generated nilpotent group.  Then $A$ satisfies the Dixmier-Moeglin equivalence.\label{thm:DMEex}
\end{theorem}
\begin{remark} We note that the group algebra of a polycyclic-by-finite group is noetherian, but a finitlely generated prime algebra satisfying a polynomial identity need not be; however, it is the case that such algebras satisfy the Goldie conditions required for localization and so one can still make sense of the rationality property in this setting.
\end{remark}

\begin{proof}[Proof of Theorem \ref{thm:DMEex}] The fact that $A$ satisfies the Nullstellensatz follows from Theorem \ref{thm:null}.  We first consider the polynomial identity case.  Here, a theorem of Posner \cite[Theorem 23.33]{Rowenbook} shows that every nonzero ideal of $A/P$ intersects the centre of $A/P$ non-trivially and hence $A/P$ is simple.  In particular $P$ is locally closed since it is maximal.  In the case of a group algebra of a finitely generated nilpotent group, we have the same result that every nonzero ideal of $A/P$ intersects the centre of $A/P$ (this result is well known and probably very old---see \cite[Proposition 3.3]{BL14} for a proof of a more general result).  Thus in both cases we obtain the result from Theorem \ref{thm:max}.  
\end{proof}
 
After the work of Dixmier and Moeglin, Goodearl and Letzter \cite{GoLet} showed that the Dixmier-Moeglin equivalence was a truly general phenomenon that applied to a large class of algebras. Their groundbreaking paper, in particular showed that it applied to many classes of quantized enveloping algebras and quantized coordinate rings of affine varieties. In particular, many of these algebras also enjoy a Hopf structure and the author and Leung \cite{BL14} asked whether the Dixmier-Moeglin equivalence holds for a complex noetherian fintiely generated Hopf algebra of finite Gelfand-Kirillov dimension.  In addition, this paper proves the result in the cocommutative case, which in light of Gromov's theorem on groups of polynomial growth (see \cite[Theroem 11.1]{KL}) can be seen as a unification of the results of Dixmier and Moeglin and work on the Dixmier-Moeglin equivalence for finitely generated nilpotent-by-finite groups (although, the proof is done via an induction, in which the enveloping algebra case is the base case).  Brown and Gilmartin \cite{BGil} ask a more modest question: does the Dixmier-Moeglin equivalence hold for pointed affine noetherian Hopf algebras of finite Gelfand-Kirillov dimension over an algebraically closed field $k$?  
 As far as Hopf algebras outside of these examples given above, the Dixmier-Moeglin equivalence has been proved for many Hopf algebras of small Gelfand-Kirillov dimension \cite{GZ}.  In addition, Brown, O'Hagan, Zhang, and Zhuang \cite{BOZZ} give the notion of an iterated Hopf Ore extension and it is known \cite{BSM18} that if $R$ is an affine commutative Hopf algebra that is an integral domain and $S=R[x;\sigma,\delta]$ has a Hopf algebra structure extending that of $R$ then $S$ satisfies the Dixmier-Moeglin equivalence.

\section{Goodearl-Letzter stratification and quantum algebras}
Arguably the greatest leap forward in work around the Dixmier-Moeglin equivalence was work of Goodearl and Letzter \cite{GoLet}, who showed that the Dixmier-Moeglin equivalence is a truly general phenomenon that holds for many classes of algebras.  In particular, they showed that algebras with a suitable rational algebraic group action will satisfy the equivalence provided that the collection of $G$-invariant prime ideals is finite.  Many quantum algebras fall into this framework.

We let $k$ be a field, let $A$ be a $k$-algebra, and let $G$ be an algebraic group over $k$ that acts as $k$-algebra automorphisms of $A$.  If for every $a\in A$ there is a finite-dimensional $k$-vector subspace $V$ of $A$ that contains $a$ such that $g\cdot V=V$ for every $g\in G$ such that the induced map $g\mapsto \Phi_g : V\to V$ is a morphism from $G\to {\rm GL}(V)$, then we say that $G$ acts \emph{rationally} on $A$. 

Given a prime ideal $P\in {\rm Spec}(A)$, we let $$J(P):=\bigcap_{g\in G} gP.$$  Then $J(P)$ is an intersection of prime ideals and if $A$ is left noetherian, we then see that $J(P)$ is a finite intersection of prime ideals $Q_1\cap \cdots \cap Q_d$.  Notice that $G$ acts transitively on the set $\{Q_1,\ldots, Q_d\}$ and since each $g\in G$ permutes the $Q_i$, and so this ideal $J(P)$ is $G$-invariant and a $G$-prime ideal; that is, if $L$ and $I$ are $G$-invariant ideals that contain $J(P)$ and $LI\subseteq J(P)$ then either $L$ or $I$ is equal to $J(P)$.

\begin{remark} Let $A$ be a left noetherian $k$-algebra with a rational action of an algebraic group $G$ and let $V$ be a finite-dimensional $G$-invariant subspace of $A$ that contains a generating set for an ideal $Q$ and let $W=Q\cap V$.  Then $\{g\in G\colon g\cdot W=W\}$ is a Zariski closed subset of $G$. In particular, the set of $g\in G$ such that $g\cdot Q=Q$ is a Zariski closed subset of $G$, and if $G$ is connected then a $G$-prime ideal is prime.  
\end{remark}
\begin{proof} Let $w_1,\ldots ,w_d$ be a basis for $W$.  Then we have map $f_i: G\to \wedge^{d+1}(V)$ given by $f_i(g)=g\cdot w_i \wedge w_1\wedge w_2\wedge \cdots \wedge w_d$. Then each $f_i$ is a morphism of algebraic varieties and $X:=\{g\in G\colon g\cdot W=W\}$ is just the intersection of $f_i^{-1}(0)$ for $i=1,\ldots ,d$.  In particular $X$ is Zariski closed.  Then we see that $X=\{g\in G\colon g\cdot W=W\}$ and so the set of $g\in G$ which preserve $Q$ is Zariski closed.  If $J$ is a $G$-prime ideal of $A$ then it is semiprime and since $A$ is left noetherian, $J$ is a finite intersection of prime ideals:
$$J=Q_1\cap \cdots \cap Q_d.$$  Then $G$ acts transitively on the $Q_i$ and if we let $h_i\in G$ be such that $h_i\cdot Q_1 = Q_i$ and we let $X=\{g\in G\colon g\cdot Q_1=Q_1\}$, then we see that $X_i := h_iX$ is Zariski closed for $i=1,\ldots ,d$ and $X_i=\{g\in G\colon g\cdot Q_1=Q_i\}$.  In particular, $G$ is a disjoint union of $d$ non-empty Zariski closed sets and since $G$ is connected, $d=1$; thus $J$ is prime.
\end{proof}

 We let $G$-${\rm Spec}(A)$ denote the set of $G$-invariant $G$-prime ideals, and for each $G$-invariant prime ideal $J$ we let 
$${\rm Spec}_J(A) = \{P\in {\rm Spec}(A)\colon J(P)=J\}.$$
\begin{theorem} (Goodearl-Letzter \cite{GoLet}) Let $k$ be an algebraically closed field and let $A$ be a noetherian $k$-algebra and that $G$ is a $k$-affine algebraic group that acts rationally on $A$. If $G$-${\rm Spec}(A)$ is finite then $A$ satisfies the Dixmier-Moeglin equivalence and for a $G$-invariant prime ideal $J$, the primitive ideals in ${\rm Spec}_J(A)$ are precisely the prime ideals that are maximal in ${\rm Spec}_J(A)$.  
\end{theorem}
We note that work of Moeglin-Rentschler \cite{MR1, MR2} and Vonessen \cite{Von1, Von2} shows that under these hypotheses that $G$ acts transitively on the rational prime ideals in each stratum.  As an example of a quick application of the Goodearl-Letzter thoerem and its great power, we observe that if  
$A$ is a connected affine noetherian $\mathbb{N}$-graded $k$-algebra over an algebraically closed field $k$, then $A$ has a $k^*$-action.  Suppose that the collection of homogeneous prime ideals of $A$ is finite.  Then $A$ has a rational $k^*$-action, by declaring that $\lambda\cdot a=\lambda^d a$ for $\lambda\in k^*$ and $a$ homogeneous of degree $d$.  Then we see that if $A$ has only finitely many homogeneous prime ideals then $A$ satisfies the Dixmier-Moeglin equivalence.  

We illustrate the Goodearl-Letzter stratification process in the case of the quantum affine $n$-space, $A=\mathbb{C}_q[x_1,\ldots ,x_n]$; that is, the algebra generated by $x_1,\ldots ,x_n$ with $x_ix_j=qx_j x_i$.  We assume that $q\in \mathbb{C}$ is nonzero and not a root of unity.  In this case we have an action of $(\mathbb{C}^*)^n$ given by $(\lambda_1,\ldots ,\lambda_n)\cdot x_i = \lambda_i x_i$.  Notice that since the variables skew commute, every element of $A$ can be written as a linear combination of monomials $x_1^{p_1}\cdots x_n^{p_n}$ and these monomials form a $\mathbb{C}$-basis for $A$.  Notice that given a monomial $m=x_1^{p_1}\cdots x_n^{p_n}$ we have $\mathbb{C}m$ is $G$-invariant and the torus action induces a $\mathbb{Z}^n$-grading of $A$ and $x_1^{p_1}\cdots x_n^{p_n}$ has degree $(p_1,\ldots ,p_n)$.  In particular, the only homogeneous elements of $A$ are scalar multiples of monomials and so any $G$-invariant ideal is generated by monomials.  Moreover, since the variables $x_1,\ldots ,x_n$ are normal (that is, $x_i A=Ax_i$ for $i=1,\ldots ,n$), we see that if $P$ is a $G$-invariant prime ideal and $x_1^{p_1}\cdots x_n^{p_n}\in P$ then there is some $i$ with $p_i>0$ such that $x_i\in P$.  In particular, each $G$-invariant prime ideal is generated by a subset of $\{x_1,\ldots ,x_n\}$.  It is straightforward to check that each ideal generated by a subset of these variables is prime, since the quotient is isomorphic to a smaller quantum affine space.  Thus there are precisely $2^n$ $G$-invariant prime ideals and so the Goodearl-Letzter theorem shows that $A$ satisfies the Dixmier-Moeglin equivalence.  Observe that if $J=\langle\{x_s \colon s\in S\}\rangle$ for some subset $S$ of $\{1,\ldots ,n\}$ then a prime ideal $P$ is in ${\rm Spec}_J(A)$ if $P$ contains $J$ and if $x_i\not \in P$ for $i\not \in S$.  Since $A/J$ is a quantum affine space, it suffices to consider the case when $J=(0)$ in $\mathbb{C}_q[x_1,\ldots ,x_m]$.  Then by the above remarks, ${\rm Spec}_J(A)$ is homeomorphic to 
the prime spectrum of $\mathbb{C}_q[x_1^{\pm 1},\ldots ,x_m^{\pm 1}]$.  In this case, they are able to show more.  Namely that the centre of this localization controls the spectrum; that is, each stratum is homeomorphic to a complex torus.  

In fact, these results apply to a large class of quantum groups (see \cite[Chapter II.8]{BrGo} for more details).  
\subsection{CGL extensions}
The work of Goodearl and Letzter shows that many iterated skew polynomial extensions, today called CGL extensions, where CGL stands for Cauchon-Goodearl-Letzter, satisfy the Dixmier-Moeglin equivalence.
An iterated skew polynomial extension
$$k[x_1][x_2;\sigma_2,\delta_2]\cdots [x_n;\sigma_n;\delta_n]$$ is called a CGL extension if the following conditions hold:
\begin{enumerate}
\item If we let $A_m=k[x_1][x_2;\sigma_2,\delta_2]\cdots [x_m;\sigma_m;\delta_m]$ for $m\le n$, then each $\sigma_m$ is a $k$-algebra automorphism of $A_{m-1}$ and $\delta_m$ is a $k$-linear locally nilpotent $\sigma_m$-derivation of $A_{m-1}$;
\item for $m=1,\ldots ,n$, $\sigma_m \circ \delta_m = q_m \delta_m \circ\sigma_m$ for some $q_m\in k^*$;
\item for $i<j$, $\sigma_j(x_i)=c_{i,j} x_i$ for some $c_{i,j}\in k^*$;
\item there is a torus $G=(k^*)^p$ that acts rationally on $A$ by $k$-algebra automorphisms;
\item the $x_i$ are $G$-eigenvectors; that is, $g\cdot x_i = \lambda_i(g)x_i$ for some character $\lambda_i: G\to k^*$ of $G$ for $i=1,\ldots ,n$;
\item there are elements $g_i\in G$ for $i=1,\dots, n$ such that $g_j\cdot x_i = \sigma_j(x_i)$ for $j>i$.
\end{enumerate}
This is a broad class of algebras that includes many quantum algebras at non-roots of unity.  For example, the quantized coordinate ring of $2\times 2$ matrices, $\mathcal{O}_q(M_2)$, is the $\mathbb{C}$-algebra with generators $a,b,c,d$ and relations
$$ab=qba, cd=qdc, ac=qca, bd=qdb, bc=cb, ad-da=(q-q^{-1})bc.$$
Then we can write this as 
$k[b][c][d;\sigma][a;\tau,\delta]$, where $\sigma(b)=qb$, $\sigma(c)=qc$, $\tau(b)=qb$, $\tau(c)=qc$,
$\tau(d)=d$ and $\delta(b)=\delta(c)=0$ and $\delta(d) = (q-q^{-1}) bc$.  Then $\delta^2(d) =0$ and so we see that $\delta$ is locally nilpotent since $\delta^2$ annihilates a generating set for $k[b][c][d;\sigma]$.  We have a torus action of $(k^*)^4$ on $A$ given by $(\alpha_1,\alpha_2,\beta_1,\beta_2)\cdot u_{i,j} = \alpha_i\beta_j u_{i,j}$ for $i,j=1,2$, where $u_{1,1},u_{1,2},u_{2,1},u_{2,2}$ are respectively $a,b,c,d$. (This can be thought of as coming from putting our generators in a $2\times 2$ matrix 
\[ \left( \begin{array}{cc} a & b \\ c & d \end{array}\right) \] and notice that relations still hold after scaling a given row by a scalar or a given column by a scalar.) We note that this action is morally a $(k^*)^3$-action since the $1$-dimensional subgroup $\{(\alpha,\alpha,\alpha^{-1},\alpha^{-1})\colon \alpha\in k^*\}$ acts trivially on $A$.  Then with this torus action it is easily seen that the ring of quantum $2\times 2$ matrices is a CGL extension.

Much of the literature involving skew polynomial extensions $R[x;\sigma,\delta]$ focuses on the case when either $\delta=0$ (i.e., the automorphism/endomorphism case) or the case when $\sigma$ is the identity (i.e., the pure derivation case).  The main reason for this is that when one leaves these cases, understanding the structure theory of the skew polynomial ring becomes considerably more difficult. One of the amazing features of CGL extensions, despite the fact that they involve both automorphisms and $\sigma$-derivations, is that, for the purposes of understanding their spectra, one can pass to the pure automorphism case via the so-called deleting derivations algorithm, which was first given by Cauchon \cite{Cau}; the main idea behind this algorithm is that by passing to a suitable localization one can show that the resulting algebra is isomorphic to a skew polynomial algebra in which one has only automorphisms and no derivations; then, by working backwards, one can reconstruct the spectrum in a step-by-step manner.  This, in particular, has allowed one to go in many cases well beyond the understanding of the prime and primitive spectrum affording by knowing the Dixmier-Moeglin equivalence and this technique has since been used to great effect by numerous authors---for a small sample, see \cite{Cas, GLL, LM11, Yak1, Yak2}.
 
\section{The Dixmier-Moeglin equivalence for related algebras}
One of the signs of how robust an algebraic property is the extent to which the property is stable when one passes to closely related algebras.  We study a few of these properties.

\subsection{Letzter's work}
One of the earliest results in this direction with the Dixmier-Moeglin equivalence is work of Letzter \cite{Let89}, which is of huge importance in understanding the Dixmier-Moeglin equivalence.  In his setting, he studies algebras $R\subseteq S$ in which $S$ is a finite left or right $R$-module.  Here we can ask to what extent properties of the zero ideal being rational, locally closed, and being primitive (when $R$ and $S$ are both prime) can be transferred from one ring to the other.  Letzter's paper is a great resource for anyone dealing with relating the rationality, locally closed, or primitive properties to rings where one ring is a finite module over the other; but we focus our attention on two of the main highlights of his work.

\begin{theorem} (Letzter \cite{Let89}) Let $k$ be a field and let $R\subseteq S$ be $k$-algebras with $S$ finitely generated on both sides as an $R$-module.  If primitive ideals of $R$ are rational then primitive ideals of $S$ are rational and if rational ideals of $R$ are primitive then rational ideals of $S$ are primitive. \end{theorem}

\begin{theorem} (Letzter \cite{Let89}) Let $k$ be a field and let $R\subseteq S$ be $k$-algebras with $S$ finitely generated and free as a left and right $R$-module. If $S$ has finite GK dimension then all primitive ideals of $R$ are locally closed in ${\rm Spec}(R)$ if and only if all primitive ideals of $S$ are locally closed in ${\rm Spec}(S)$.
\end{theorem}
Notice that in the case that $R\subseteq S$ are $k$-algebras with $S$ a finitely generated free left and right $R$-module and $S$ noetherian of finite GK dimension and $R$ and $S$ satisfy the Nullstellensatz then we have that $R$ satisfies the Dixmier-Moeglin equivalence if and only if $S$ does.

As an application of Letzter's results, we prove the following folklore theorem, which we attributed to Zalesski\u\i, although it was apparently not stated in this form.
\begin{theorem} Let $G$ be a finitely generated nilpotent-by-finite group and let $k$ be a field.  Then $k[G]$ satisfies the Dixmier-Moeglin equivalence.\label{thm:nbf}
\end{theorem}
\begin{proof} Let $N$ be a finite-index nilpotent subgroup of $G$. Then $S=k[G]$ is a finitely generated free left and right $R=k[N]$-module. Then $S$ is noetherian and satisfies the Nullstellensatz and has finite GK dimension by the Bass-Guivarch theorem (see Krause and Lenagan \cite[Theorem 11.14]{KL}).  Then by Theorem \ref{thm:DMEex} we have $R$ satisfies the Dixmier-Moeglin equivalence, so by Letzter's thoerems we have that $S$ does too.
\end{proof}
\subsection{Irving-Small reduction}
When studying algebras over a field $k$ it is often much simpler to deal with the case that $k$ is algebraically closed.  Let $A$ be a $k$-algebra and let $\bar{k}$ denote the algebraic closure of $A$.  Suppose that $A\otimes_k \bar{k}$ is noetherian and satisfies the Nullstellensatz. 
\begin{theorem} (Irving-Small) Let $k$ be a field of characteristic zero, let $A$ be a finitely generated $k$-algebra, and suppose that $A\otimes_k K$ is left noetherian and satisfies the Nullstellensatz for every extension $K$ of $k$.  If rational ideals are primitive in $A\otimes_k \bar{k}$ and primitive ideals are locally closed in $A\otimes_k K$ for some algebraically closed uncountable field extension of $k$ then the Dixmier-Moeglin equivalence holds for $A$.
\end{theorem}
\begin{proof} This exact form is not found in Irving-Small and one should see Rowen \cite[Theorem 8.4.27]{Row2}.
\end{proof}
This argument was how Irving-Small showed that the Dixmier-Moeglin equivalence holds for $U(L)$ when $L$ is a finite-dimensional Lie algebra over a field $k$ of characteristic zero.  The original work of Dixmier and Moeglin applies to the complex case, although their argument applies to finite-dimensional Lie algebras over an uncountable algebraically closed field of characteristic zero.  Since enveloping algebras behave well under base change, one can then invoke the Irving-Small result above to get that the Dixmier-Moeglin equivalence holds over a field of characteristic zero.  We note that in positive characteristic, enveloping algebras of finite-dimensional Lie algebras satisfy a polynomial identity \cite{Bac}, and so the Dixmier-Moeglin equivalence holds here too, by Theorem \ref{thm:DMEex}.
\subsection{Other closure properties}
 Bell, Wu, Wu \cite{BWW} gave general closure properties under certain classes of Ore extensions and extension of scalars.  In particular if $k$ is an uncountable algebraically closed field of characteristic zero and if $A$ is a finitely generated noetherian $k$-algebra of finite GK dimension such that all prime ideals of $A$ are completely prime, then if $A$ satisfies the Dixmier-Moeglin equivalence then so do the Ore extensions $A[x;\sigma]$, $A[x;\delta]$, where $\sigma$ is a $k$-algebra automorphism of $A$ and $\delta$ is a $k$-linear derivation of $A$, provided $\sigma$ and $\delta$ preserve a finite-dimensional generating subspace of $A$.  It would be good to remove the hypotheses on the base field and the completely prime hypothesis on prime ideals in this result, although the completely prime hypothesis does cover the cases of a commutative base ring, which is perhaps the most interesting case when dealing with skew polynomial rings.  
 
 In terms of extensions of scalars, the authors show that if $k$ is an uncountable algebraically closed field of characteristic zero, then if $A$ is a noetherian algebra that satisfies the Dixmier-Moeglin equivalence then so is $A\otimes_k F$ for every extension of $F$, provided ${\rm dim}_k(A)\le \aleph_0$.  We note that in general one does not have that the noetherian property is preserved under extension of scalars, so this result does require some sort of base field hypothesis.  Extending this, recent work of Bell, Wang, and Yee  \cite{BWY} shows that if $R$ and $S$ are prime noetherian algebras that satisfy the Dixmier-Moeglin equivalence and if $R\otimes_k S$ is prime noetherian and satisfies the Nullstellensatz then $R\otimes_k S$ satisfies the Dixmier-Moeglin equivalence.  This then extends the result of \cite{BWW}.  In addition to this, \cite{BWY} shows that satisfying the Dixmier-Moeglin equivalence is a Morita invariant and that if $R$ satisfies the Dixmier-Moeglin equivalence and $e$ is an idempotent of $R$ then $eRe$ satisfies the Dixmier-Moeglin equivalence.

\section{Lorenz' extension to the non-noetherian setting}
\label{LorSec}
One very interesting development in the study of the Dixmier-Moeglin equivalence has been work of Lorenz \cite{Lor08, Lor09}, extending this to the non-noetherian case. We note that the definition of rationality requires the ability to form some sort of localization and this is not available in general.  Lorenz works around this issue, by working with the extended centroid, which has the advantage of being something that can be formed in any associative algebra and coincides with the centre of the Goldie ring of quotients in the prime Goldie case.

To do this, one constructs the Amitsur-Martindale ring of quotients of $R$, which we denote $Q_{AM}(R)$.  We give an overview of the construction, but only in the case of a prime ring.  For more details about the general construction and its properties, we refer the reader to \cite{Lor08}.  

Given a prime ring $R$ and a nonzero two sided ideal $I$ of $R$, we let ${\rm Hom}(I_R,R_R)$ denote the set of all right $R$-module homomorphisms from $I$ to $R$.  Then we take $X$ to be the union of the hom-sets 
${\rm Hom}(I_R,R_R)$ as $I$ ranges over nonzero two-sided ideals of $R$ and we put an equivalence relation $\sim$ on $X$ by declaring that $f:I\to R$ and $g:J\to R$ are equivalent if the restrictions of $f$ and $g$ to some nonzero ideal $L\subseteq I\cap J$ agree.  Since $R$ is prime, the intersection of two nonzero ideals is again nonzero.  Then we take $Q_{AM}(R)$ to be $X/\sim$.  Then we can add elements $f:I\to R$ and $g:J\to R$ of $Q_{AM}(R)$ by adding their restrictions to $I\cap J$ and multiplication is given by composition, where to form $f\circ g$, we restrict $g$ to $g^{-1}(I)$.  It is straightforward, although somewhat tedious, to check that these operations respect $\sim$ and that $Q_{AM}(R)$ is a ring under these operations.  

We observe that $R$ embeds in $Q_{AM}(R)$ via the homomorphism $r\mapsto f_r :R_R\to R_R$, $f_r(x)= rx$. This is an embedding, because if it were not then we would necessarily have $f_r\sim 0$ and so there would be a nonzero two-sided ideal $I$ such that $f_r(I)=0$. But this cannot occur unless $r=0$ since $R$ is prime.  We then define the \emph{extended centroid}, $\mathcal{C}(R)$, of $R$ to be $Z(Q_{AM}(R))$.  Observe that $\mathcal{C}(R)$ is a field when $R$ is a prime ring since if $z: I_R\to R_R$ is a nonzero central element of $Q_{AM}(R)$, then we first note that $z$ is $1$-$1$ since if $z(a)=0$ for some nonzero $a\in I$ then $z$ vanishes on the right ideal $aR$; since $z$ commutes with the maps $f_r$ above then one can show that $z$ vanishes on a two-sided ideal, which contradicts the fact that it is nonzero. If we now let $J=z^{-1}(I)$ and define $y:J_R\to R_R$ via $y(r) = a$ where $a\in R$ is the unique element such that $z(a)=x$ then we see that $y$ is a right $R$-module homomorphism that is the inverse of $z$.

Thus $\mathcal{C}(R)$ is a field and we say that a prime $k$-algebra is \emph{rational} if $\mathcal{C}(R)$ is an algebraic extension of $R$.  When $R$ is prime right Goldie (in particular, if $R$ is prime noetherian) then $Z(Q(R))$ coincides with $\mathcal{C}(R)$ (see \cite[1.4.2]{Lor08}).  We note that in the prime noetherian case, if $z=ab^{-1}\in Z(Q(R))$, with $b$ regular, then the fact that $[z,b]=0$ gives that $a$ and $b$ commute and so $ab^{-1}=b^{-1}a$ and then $[z,r]=0$ gives $arb=bra$ for all $r\in R$. In particular if we let $J=RbR$ then the map $f_z:J_R\to R_R$ given by $f_z(xby) = xay$ is a right $R$-module homomorphism, since if $\sum x_i b y_i =0$ then $\sum x_i b y_i a =0$, which then gives
$\sum x_i a y_i b =0$ using the fact that $arb=bra$ for all $r\in R$.  Finally, since $b$ is regular, this gives $\sum x_i a y_i =0$ and so we see that $f_z$ is indeed a well-defined homomorphism from $J_R$ to $R_R$. 
Then the map $z\mapsto f_z$ gives an injection from $Z(Q(R))\to \mathcal{C}(R)$.  Moreover, given $g:J_R\to R_R$ in $\mathcal{C}(R)$ we can pick some regular $b\in J_R$ if $R$ is prime noetherian and if we let $z = g(b) b^{-1}$ then the fact that $g$ is central in $Q_{AM}(R)$ gives that $\sum x_i b y_i =0 \implies \sum x_i g(b) y_i =0$.  

Lorenz works in this setting of having a potentially non-noetherian $k$-algebra $R$ over an algebraically closed field $k$ endowed with a rational action of an algebraic group $G$ over $k$, acting by $k$-algebra automorphisms of $R$.  Then $G$ acts on ${\rm Spec}(R)$, and one can see that the action of $G$ preserves the set of locally closed ideals, the set of primitive ideals, and the set of rational prime ideals of $R$.  As before, we let $G$-${\rm Spec}(R)$ denote the set of $G$-prime ideals; that is, $G$-invariant ideals $I$ such that if $JL\subseteq I$ with $J,L$ $G$-invariant ideals containing $I$ then either $J=I$ or $L=I$.  Then we have a map $\Phi$ from ${\rm Spec}(R)$ to $G$-${\rm Spec}(R)$ given by
$$P\mapsto \bigcap_{g\in G} g\cdot P.$$  We note that $\Phi(P)$ is a semiprime ideal of $R$ and since $\Phi(P)$ is $G$-invariant, $G$ acts on $R/\Phi(P)$ and this induces an action of $G$ on the extended centroid of $R/\Phi(P)$.  We then say that $\Phi(P)$ is $G$-rational if the $G$-invariant elements of $\mathcal{C}(R/\Phi(P))$ is an algebraic extension of $k$.  We deal with the case when $k$ is algebraically closed, and so in this case the $G$-invariants should just be $k$.  

Remarkably, Lorenz \cite{Lor08} shows that this map takes rational prime ideals of $R$ to $G$-rational $G$-
prime ideals and even more striking is that this map gives a bijection between $G$-orbits of rational prime ideals of $R$ and $G$-rational $G$-prime ideals of $R$.  In a sequel to this paper, Lorenz \cite{Lor09} shows that when $R$ satisfies the Nullstellensatz, we have that the collection of $G$-prime ideals is finite if and only if $R$ satisfies the Dixmier-Moeglin equivalence. What is truly remarkable about this result is how general it is: for example, it applies to any affine algebra over the complexes that is endowed with a rational action of an algebraic group.  In a later paper, Lorenz \cite{Lor14} shows that when the group is a torus then one has that the prime spectrum can be expressed as a disjoint union of spaces that are homeomorphic to tori, which is a generalization of an important result of Goodearl and Letzter \cite{GoLet}.
Lorenz' extension of the Dixmier-Moeglin equivalence to not-necessarily-noetherian algebras is used in work of Abrams, Rangaswamy, and the author \cite{ABR} to show that the Dixmier-Moeglin equivalence holds, in the above sense, for Leavitt path algebras of finite directed graphs, by exploiting a natural torus action.  
\subsection{The Poisson Dixmier-Moeglin equivalence}
Quantum algebras and Poisson algebras share an intimate connection.  This comes from the fact that if one has a quantized coordinate ring of an affine variety $\mathcal{O}_q(V)$ with $q$ a transcendental indeterminate then when one specializes $q$ to be $1$, one recovers the ordinary coordinate ring of $V$, a commutative ring.  In particular, it is convenient to think of $\mathcal{O}_q(V)$ as a faithfully flat $\mathbb{Z}[q,q^{-1}]$-algebra tensored up to $\mathbb{C}$, where we can then specialize $q$ at nonzero complex values, and we concentrate now on this setting.  Then for $x,y\in \mathcal{O}_q(V)$, one has $[x,y]=(q-1)f(x,y)$ for some $f(x,y)\in \mathcal{O}_q(V)$ and so we can create a bracket $\{\cdot \, , \cdot \}: \mathcal{O}(V)\times \mathcal{O}(V)\to \mathcal{O}(V)$ by declaring that $\{x,y\} = f(x,y)|_{q=1}$; i.e., the image of $f(x,y)$ in $\mathcal{O}(V)$ after specializing at $q=1$.  For example, if we take the ring of $2\times 2$ quantum matrices, which has generators $a,b,c,d$ and relations $ab=qba, ac=qca, cd=qdc, bd=qdb, bc=cb, ad-da=(q-q^{-1})bc$, then if we apply the procedure given above, we find
$$ab-ba = (q-1)ba, ac-ca = (q-1)ca, cd-dc=(q-1)dc,$$~$$bd-db=(q-1)db, bc-cb = (q-1)\cdot 0, ad-da = (q-1)(q+1)q^{-1} bc,$$ and so
$$\{a,b\}=ba, \{a,c\}=ca, \{c,d\}=dc, \{b,d\}=db, \{b,c\}=0, \{a,d\}= 2bc.$$
As it turns out, knowing what the bracket does to generators tells you how to commute the bracket for any two elements of the algebra.   The reason for this is that the bracket is easily seen to be $k$-bilinear, anti-symmetric, and one can check that for $f,g,h\in \mathcal{O}(V)$ and $\{fg, h\} = \{f,h\}g +f\{g,h\}$.  In fact, this bracket is a \emph{Poisson bracket}, meaning that it is a Lie bracket and has the property that for each fixed $f\in \mathcal{O}(V)$ the maps $L_f, R_f: \mathcal{O}(V)\to \mathcal{O}(V)$ given by $L_f(g) = \{f,g\}$ and $R_f(g)=\{g,f\}$ are $k$-linear derivations of $\mathcal{O}(V)$.  Thus we have the principle that quantizations of classical (commutative) objects should generally yield a Poisson bracket on the commutative object. A remarkable result due to Kontsevich \cite{K} shows that a type of reverse principle holds: Given a Poisson manifold one can create a family of deformations of the space of smooth functions on the manifold parametrized by a variable $\hbar$.

In light of this correspondence and the work of Goodearl and Letzter mentioned earlier, it is natural to ask whether there is a Poisson Dixmier-Moeglin equivalence for affine commutative algebras equipped with a Poisson bracket.  To make sense of this, we must define the notions of rationality, primitivity, and being locally closed in the Poisson setting.  Let $k$ be a field of characteristic zero, and let $A$ be a finitely generated commutative $k$-algebra equipped with a Poisson bracket $\{\cdot \, , \cdot \}$.  We call such an algebra a \emph{Poisson algebra}.  Then an ideal $I$ of $A$ is a \emph{Poisson ideal} if for $f\in I$ and $g\in A$ we have $\{f,g\}\in A$.  We note that if $I$ is a Poisson ideal then it is not difficult to show that so is its radical and so is every minimal prime ideal above $I$.  We call a prime ideal that is a Poisson ideal a \emph{Poisson prime ideal}.  Then if $P$ is a Poisson prime ideal, we can form the quotient ring $B=A/P$ and the Poisson bracket will induce a Poisson bracket on $B$, and we can extend this bracket to the field of fractions of $B$ via the rule $\{fg^{-1},h\} = \{f,h\} g^{-1} - f g^{-2} \{g,h\}$ for $f,g,h\in B$ and then use the anti-symmetric property to extend this in the case when $h$ is in ${\rm Frac}(B)$.  Given an algebra $R$ with a Poisson bracket, we call the \emph{Poisson centre} of $R$ the set of $f\in R$ such that $\{f,g\}=0$ for all $g\in R$.  Then $P$ is \emph{Poisson rational} if the Poisson centre of ${\rm Frac}(B)$ is a finite extension of $k$; $P$ is \emph{Poisson primitive} if there is a maximal ideal of $Q$ of $B$ that does not contain any nonzero Poisson prime ideals; finally, $P$ is \emph{Poisson locally closed} if the intersection of the nonzero Poisson prime ideals that strictly contain $P$ is an ideal that properly contains $P$.  

Armed with these notions, we can then define the Poisson Dixmier-Moeglin equivalence for a Poisson algebra, as being an algebra for which the notions of Poisson rationality, Poisson primitivity, and being Poisson locally closed are equivalent for all Poisson prime ideals of the algebra.   Brown and Gordon \cite[Question 3.2]{BrGo2} whether the Poisson Dixmier-Moeglin equivalence holds for all affine complex Poisson algebras, and it has been shown to hold in numerous cases: mainly those coming via the semiclassical limit construction applied to a quantum algebra and closely related algebras (see, for example,  \cite{LL, J, O1, O2, GL2}).  A negative answer to the question of Brown and Gordan was given in \cite{BLLM17} although it is shown in this paper that the answer is affirmative when the Krull dimension is at most two.
  
\section{Counterexamples}
As has been pointed out already, there are counterexamples to the Dixmier-Moeglin equivalence for affine noetherian algebras.  In the non-affine case it is quite easy to construct counterexamples.  Nevertheless, there are no ``easy'' affine counterexamples, in the sense that it generally requires some effort to show that a given ring does not satisfy the Dixmier-Moeglin equivalence---especially for algebras which satisfy the Nullstellensatz. The reason for this is that rational prime ideals, even when not locally closed, are in some sense close to being locally closed.  In the affine case there are at most countably many height one primes when $(0)$ is a rational prime ideal of a noetherian ring (cf. Irving \cite{Irv79a}).  And showing that there are indeed infinitely many such primes can be non-trivial.  

Most counterexamples rely on characterizations of primitivity in skew polynomial rings given in \S \ref{skew}, and are generally expressible as a skew polynomial extension of an commutative ring.  The first counterexample is due to Lorenz \cite{Lor}. Lorenz constructed an algebra $R=\mathbb{C}[x^{\pm 1},y^{\pm 1}][z^{\pm 1};\sigma]$ where $\sigma$ is an automorphism of the form
$x\mapsto x^a y^b, y\mapsto x^c y^d$, where
\[ A:=\left( \begin{array}{cc} a & b \\ c & d\end{array}\right)\]
is a matrix in ${\rm SL}_2(\mathbb{Z})$ whose eigenvalues are not roots of unity.  Then $R$ is the complex group algebra of a polycyclic group of the form $\mathbb{Z}^2 \rtimes \mathbb{Z}$.  Then Lorenz shows that $(0)$ is a rational prime ideal in this case, but that it is not locally closed.  Characterizations of primitivity in skew Laurent polynomial rings give that $R$ is primitive. 
So his example shows that $${\rm rational},~{\rm primitive}\centernot\implies {\rm locally~closed}$$ in general.

Irving \cite{Irv01} constructed an example of a primitive noetherian algebra with centre not equal to a field.  In particular, this algebra does not satisfy the Nullstellensatz and it is immediate that for such a ring $(0)$ is a primitive ideal that is not rational.  Thus we see that $${\rm primitive} \centernot\implies {\rm rational}$$ in general for affine noetherian algebras.  Earlier counterexamples due to Irving can be round in \cite{Irv79}

Jordan \cite[7.10--7.14]{Jor93} gives the following construction. 
Let $S = R[t;\sigma]$, where $R = \mathbb{C}[x, x^{-1}, y, y^{-1}]$ and where $\sigma$ is the $\mathbb{C}$-algebra automorphism given by $x\mapsto yx^{-1}$, $y\mapsto x$.  Jordan claims that the ring $R$ is not $\sigma$-special and so the L\'eroy-Matczuk theorem then gives that $S$ is not primitive.  On the other hand, it is not hard to show that $(0)$ is a rational prime ideal that is not locally closed.  It should be noted that there is a gap in Jordan's argument, which is filled in a paper of Brown, Carvalho, and Matczuk \cite{BCM17}, using an argument due to Goodearl that relies on deep techniques from the field of ``unlikely intersections''.  A sketch of a more elementary argument is also provided by the authors and a proof of a more general result appears in \cite{BG18}.  Thus we see that 
$${\rm rational}\centernot\implies {\rm primitive}.$$

Bell, Launois, Le\'on S\'anchez, and Moosa \cite{BLLM17} constructed a counterexample to the Poisson Dixmier-Moeglin equivalence, which also yields a counterexample to the ordinary Dixmier-Moeglin equivalence with a ring of the form $R[x;\delta]$ with $R$ a finitely generated complex commutative algebra of Krull dimension three.  This example is notable in that it has finite Gelfand-Kirillov dimension.  We are unaware of any counterexamples to the Dixmier-Moeglin equivalence that are affine noetherian algebras with Gelfand-Kirillov dimension at most $3$. 

We end this survey with a question.  Notice that the above examples show that in the affine noetherian case, 
that none of the various implications comprising the Dixmier-Moeglin equivalence hold in general, with the possible exception that locally closed might imply primitive or rational.  We ask whether it is the case that a prime $P$ of an affine noetherian ring being locally closed in the prime spectrum implies that it is either rational or primitive?

\end{document}